\documentclass[12pt,a4paper,oneside]{amsart}
\usepackage[a4paper]{geometry}
\usepackage{palatino}

\usepackage{amssymb,enumerate,perpage,mathrsfs}

\DeclareMathOperator{\ad}{\mathsf{ad}}
\DeclareMathOperator{\Der}{Der}
\DeclareMathOperator{\Hom}{Hom}

\DeclareMathOperator{\id}{id}

\DeclareMathOperator{\sym}{S}
\DeclareMathOperator{\Tr}{Tr}

\newcommand{\vertbar}{\>|\>}
\newcommand{\set}[2]{\ensuremath{\{ #1 \vertbar #2 \}}}
\def\liebrack  {\ensuremath{[\,\cdot\, , \cdot\,]}}
\def\topbar{\ensuremath{\overline{\phantom a}}\,}

\MakePerPage[2]{footnote}

\newtheorem{theorem}{Theorem}
\newtheorem*{proposition}{Proposition}
\newtheorem*{corollary}{Corollary}
\newtheorem{lemma}{Lemma}
\begin{document}

\title{On Hermitian and skew-Hermitian matrix algebras over octonions}

\author{Arezoo Zohrabi}
\address{University of Ostrava, Ostrava, Czech Republic}
\email{azohrabi230@gmail.com}

\author{Pasha Zusmanovich}
\email{pasha.zusmanovich@osu.cz}

\date{December 6, 2019}

\begin{abstract}
We prove simplicity, and compute $\delta$-derivations and symmetric associative
forms of algebras in the title.
\end{abstract}

\maketitle

\section*{Introduction}

We consider algebras of Hermitian and skew-Hermitian matrices over octonions.
While such algebras of matrices of low order are well researched and well 
understood (the algebra of $3 \times 3$ Hermitian matrices being the famous
exceptional simple Jordan algebra), this is not so for higher orders; the case
of Hermitian matrices of order $4 \times 4$ appears in modern physical theories.

Derivation algebras of these algebras were recently computed in \cite{petyt},
and here we continue to study these algebras. After the preliminary 
\S \ref{sec-1}, where we set notation and remind basic facts about algebras with
involution, we prove simplicity of the algebras in question (\S \ref{sec-simp}),
and compute their $\delta$-derivations (\S \ref{sec-der}) and symmetric associative forms (\S \ref{sec-form}). The last \S \ref{sec-q} 
contains some further questions.

\section{Notation, conventions, preliminary remarks}\label{sec-1}

\subsection{}

The ground field $K$ is assumed to be arbitrary, of characteristic $\ne 2,3$. 
``Algebra'' means an arbitrary algebra over $K$, not necessary associative, or 
Lie, or Jordan, or satisfying any other distinguished identity, unless specified
otherwise. If $a$ is an element of an algebra $A$, $R_a$ denotes the linear 
operator of the right multiplication on $a$. All unadorned tensor products and
$\Hom$'s are over the ground field $K$. 

\subsection{Algebras with involution}

An \emph{involution} on a vector space $V$ is a linear map $j: V \to V$ such 
that $j^2 = \id_V$. If $j$ is involution on $V$, define
\begin{align*}
&\sym^{+}(V,j) = \set{x \in V}{j(x) = x} \\
\intertext{and}
&\sym^{-}(V,j) = \set{x \in V}{j(x) = -x} ,
\end{align*}
the subspaces of $j$-symmetric and $j$-skew-symmetric elements of $V$, 
respectively.

For an arbitrary vector space with an involution $j$, we have the direct sum 
decomposition:
$$
V = \sym^+(V,j) \dotplus \sym^-(V,j) .
$$

An \emph{involution} on an algebra $A$ is a linear map $j: A \to A$ which is an
involution of $A$ as a vector space, and, additionally, is an antiautomorphism
of $A$, i.e. $j(xy) = j(y)j(x)$ for any $x,y \in A$.

For an arbitrary algebra $A$ with an involution $j$, the subspace $\sym^+(A,j)$
is closed with respect to anticommutator $x \circ y = \frac 12 (xy + yx)$ and 
thus forms a (commutative) algebra with respect to $\circ$. The anticommutator 
will be also frequently referred as \emph{Jordan product}, despite that the 
ensuing algebras are, generally, not Jordan. Similarly, the subspace $\sym^-(A,j)$ is closed with respect to commutator $[x,y] = xy - yx$, 
and thus forms an (anticommutative) algebra with respect to $\liebrack$. 

We have the following obvious inclusions:
\begin{alignat}{1}\label{eq-inc}
&\sym^+(A,j) \circ \sym^+(A,j) \subseteq \sym^+(A,j) \notag \\
&\sym^+(A,j) \circ \sym^-(A,j) \subseteq \sym^-(A,j) \\
&\sym^-(A,j) \circ \sym^-(A,j) \subseteq \sym^+(A,j) \notag
\end{alignat}
and
\begin{alignat}{1}\label{eq-incbrack}
&[\sym^+(A,j), \sym^+(A,j)] \subseteq \sym^-(A,j) \notag \\
&[\sym^+(A,j), \sym^-(A,j)] \subseteq \sym^+(A,j) \\
&[\sym^-(A,j), \sym^-(A,j)] \subseteq \sym^-(A,j) . \notag
\end{alignat}

If $(A,j)$ and $(B,k)$ are two vector spaces, respectively algebras, with 
involution, then their tensor product $(A \otimes B, j\otimes k)$, is a vector
space, respectively algebra, with involution. Here $j\otimes k$ acts on 
$A \otimes B$ in an obvious way: 
$(j \otimes k) (a \otimes b) = j(a) \otimes k(b)$ for $a \in A$, $b \in B$.

\subsection{Matrix algebras}

$M_n(K)$ denotes the (associative) algebra of $n \times n$ matrices with entries
in $K$. The matrix transposition, denoted by ${}^\top$, is an involution on 
$M_n(K)$. We use the shorthand notation $M_n^+(K) = \sym^+(M_n(K),{}^\top)$ and 
$M_n^-(K) = \sym^-(M_n(K),{}^\top)$ for the spaces of symmetric and 
skew-symmetric $n \times n$ matrices, respectively. The algebra $M_n^+(K)$ with
respect to the Jordan product is a simple Jordan algebra, and the algebra 
$M_n^-(K)$ with respect to the commutator is the (semi)simple orthogonal Lie 
algebra, customarily denoted by $\mathfrak{so}_n(K)$\footnote{
$\mathfrak{so}_n(K)$ is isomorphic to $\mathfrak{sl}_2(K)$ for $n=3$, to
$\mathfrak{sl}_2(K) \oplus \mathfrak{sl}_2(K)$ for $n=4$, 
to $\mathfrak{sl}_3(K)$ for $n=6$, and is a simple Lie algebra of type $B_k$ for
$n=2k+1$, $k \ge 2$, or of type $D_k$ for $n = 2k$, $k \ge 4$, but these details
are immaterial for our considerations here.
}. 
$\Tr(X)$ denotes the trace of a matrix $X$, and $E$ denotes the identity matrix.

\begin{lemma}\label{lemma-1}
If $x \in M_n^-(K)$ is such that $x \circ M_n^-(K) = 0$, then $x=0$.
\end{lemma}

\begin{proof}
Considering this on the Lie algebra level, we have $xy + yx = 0$ for any 
$y \in \mathfrak{so}_n(K)$. Taking the trace of the both sides of this equality,
we have $\Tr(xy) = 0$. The left-hand side in the last equality is proportional 
to the Killing form, and since the Killing form on the (semi)simple Lie algebra
$\mathfrak{so}_n(K)$ is nondegenerate, we have $x=0$.
\end{proof}

\begin{lemma}\label{lemma-circ}
If $m \in M_n^+(K)$ is such that $[m,M_n^-(K)] = 0$ or $[m,M_n^+(K)] = 0$, then
$m$ is a multiple of $E$.
\end{lemma}

\begin{proof}
\emph{Case of $[m,M_n^-(k)] = 0$}. 
Inspection of tables of irreducible representations of simple classical Lie 
algebras reveals that the representation of $\mathfrak{so}_n(K)$ in $M_n^+(K)$,
being isomorphic to the symmetric square of the natural representation, 
decomposes as the direct sum of the trivial $1$-dimensional representation 
(spanned by the identity matrix), and the $\frac{n^2 + n - 2}{2}$-dimensional 
irreducible representations in the case $n \ne 4$, and the tensor product of two
irreducible representations in the case $n=4$ (see, for example, 
\cite[Lemma 3.1]{bbm}). The statement of Lemma than readily follows.

\emph{Case of $[m,M_n^+(k)] = 0$}.
It is easy to check that this condition implies 
$$
(m,s,t) = (s,m,t) = (s,t,m) = 0
$$
for any $s,t \in M_n^+(K)$, where 
$(x,y,z) = (x \circ y) \circ z - x \circ (y \circ z)$ is the Jordan associator,
i.e. $m$ lies in the center of the simple Jordan algebra $(M_n^+(K), \circ)$, 
which coincides with $KE$.
\end{proof}

\subsection{The octonion algebra}

Note that we do not assume the ground field $K$ to be algebraically closed, but 
the split octonion algebra $\mathbb O$ is defined uniquely over any field. This
is the algebra with unit $1$. Let us note the properties of its standard basis
$\{1, e_1, \dots, e_7\}$ we will need in the sequel. We have $e_i^2 = -1$, 
$e_ie_j = -e_je_i$, and, denoting by $B_i$ the $6$-dimensional linear span of 
all the basic elements except of $1$ and $e_i$, we have 
$e_i B_i = B_i e_i = B_i$, for any $i=1,\dots,7$ (see, for example, 
\cite[\S 2, Table 1]{baez}). By 
$$
*: \{1,\dots,7\} \times \{1,\dots,7\} \to \{1,\dots,7\}
$$
we denote the partial binary operation such that 
$e_ie_j = -e_je_i = \pm e_{i*j}$, $i \ne j$.

The standard conjugation in $\mathbb O$, denoted by $\topbar$, and defined by
$\overline 1 = 1$, $\overline e_i = -e_i$, turns it to an algebra with 
involution. Thus, denoting $\mathbb O^+ = \sym^+(\mathbb O,\topbar)$ and 
$\mathbb O^- = \sym^-(\mathbb O,\topbar)$, we have $\mathbb O^+ = K1$, and 
$\mathbb O^-$ is the $7$-dimensional subspace of imaginary octonions, linearly
spanned by $e_1, \dots, e_7$. The latter subspace, with respect to the 
commutator, forms the $7$-dimensional simple Malcev algebra.

As for any $a \in \mathbb O$, the elements $a + \overline a$ and $a \overline a$
belong to $\mathbb O^+$, we get the linear map $T: \mathbb O \to K$ and the 
quadratic map $N: \mathbb O \to K$, defined by $T(a) = a + \overline a$ and 
$N(a) = a\overline a$, called the \emph{trace} and \emph{norm}, respectively. 
Any element $a \in \mathbb O$ satisfies the quadratic equality
\begin{equation}\label{eq-r}
a^2 - T(a)a + N(a)1 = 0
\end{equation}
(see, for example, \cite[Chapter III, \S 4]{schafer}).

For any two elements $a,b\in \mathbb O^-$, writing the equality (\ref{eq-r}) for
the element $a+b$, subtracting from it the same equalities for $a$ and for $b$,
and taking into account that $T(a) = T(b) = 0$, yields
\begin{equation}\label{eq-o}
ab + ba = N(a,b) 1 ,
\end{equation}
where $N(a,b) = N(a) + N(b) - N(a+b)$.

\subsection{Algebras of Hermitian and skew-Hermitian matrices over octonions}

Our main characters, the algebras of Hermitian and skew-Hermitian matrices over
octonions, are defined as $\sym^+(M_n(\mathbb O),J)$ and 
$\sym^-(M_n(\mathbb O),J)$ respectively, where $M_n(\mathbb O)$ is the algebra 
of $n \times n$ matrices with entries in $\mathbb O$. The involution on 
$M_n(\mathbb O)$ is defined as $J: (a_{ij}) \mapsto (\overline{a_{ji}})$, i.e.,
the matrix is transposed and each entry is conjugated, simultaneously.

The algebras $\sym^+(M_n(\mathbb O,J))$ are unital, the identity matrix being a
unit. These algebras for small $n$'s are Jordan algebras, well-known from the 
literature: for $n=1$ this is nothing but the ground field $K$, for $n=2$ this 
is the $10$-dimensional simple Jordan algebra of symmetric nondegenerate 
bilinear form (see, for example, \cite[Chapter IX, Exercise 4]{involutions} and 
\cite[\S 6]{ruhaak}), and for $n=3$ this is the famous $27$-dimensional 
exceptional simple Jordan algebra. For $n \ge 4$, this is no longer a Jordan 
algebra, but the case $n=4$ has some importance in modern physics, see 
\cite{toppan}; interestingly enough, this case was considered already in a 
little-known dissertation \cite{ruhaak}, under the direction of Hel Braun and 
Pascual Jordan.

The algebras $\sym^-(M_n(\mathbb O,J))$ are less prominent; it seems that the 
only case which has been appeared in the literature is $n=1$: the 
$7$-dimensional simple Malcev algebra $\mathbb O^-$.

Due to the isomorphism of algebras 
$M_n(\mathbb O) \simeq M_n(K) \otimes \mathbb O$, the algebra with involution
$(M_n(\mathbb O),J)$ can be represented as the tensor product of two 
algebras with involution: $(M_n(K), {}^\top)$, the associative algebra of 
$n \times n$ matrices over $K$ with involution defined by the matrix 
transposition, and $(\mathbb O, \topbar)$.

\section{Simplicity}\label{sec-simp}

\begin{proposition}
For any two vector spaces with involution $(V,j)$ and $(W,k)$, there are 
isomorphisms of vector spaces
\begin{alignat*}{1}
&\sym^{+}(V \otimes W, j \otimes k) \simeq
\sym^{+}(V,j) \otimes \sym^{+}(W,k) \dotplus
\sym^{-}(V,j) \otimes \sym^{-}(W,k)
\\
&\sym^{-}(V \otimes W, j \otimes k) \simeq
\sym^{+}(V,j) \otimes \sym^{-}(W,k) \dotplus
\sym^{-}(V,j) \otimes \sym^{+}(W,k) .
\end{alignat*}
\end{proposition}

\begin{proof}
Let us prove the first isomorphism, the proof of the second one is completely
similar. By definition, an element $\sum_{i\in \mathbb I} v_i \otimes w_i$ of 
$V \otimes W$ belongs to $\sym^{+}(V \otimes W, j \otimes k)$, if and only if
$$
\sum_{i\in \mathbb I} \Big(J(v_i) \otimes K(w_i) - v_i \otimes w_i\Big) = 0 .
$$

Applying to this equality the linear maps $(\id_V + j) \otimes \id_W$ and 
$(\id_V - j) \otimes \id_W$, we get respectively:
$$
\sum_{i\in \mathbb I} (j(v_i) + v_i) \otimes (k(w_i) - w_i) = 0 
$$
and
$$
\sum_{i\in \mathbb I} (j(v_i) - v_i) \otimes (k(w_i) + w_i) = 0 .
$$

Applying \cite[Lemma 1.1]{low} to the last two equalities, we can replace 
$v_i$'s and $w_i$'s by their linear combinations in such a way that the index 
set is partitioned in the following way: 
$\mathbb I = \mathbb I_{11} \cup \mathbb I_{12} \cup \mathbb I_{21} \cup \mathbb I_{22}$,
where
\begin{alignat*}{5}
&v_i&\in& \sym^-(V,j), \>&v_i&\in& &\sym^+(V,j) \quad &\text{for }& i\in \mathbb I_{11}
\\
&v_i&\in& \sym^-(V,j), \>&w_i&\in& &\sym^-(W,k) \quad &\text{for }& i\in \mathbb I_{12}
\\
&v_i&\in& \sym^+(V,j), \>&w_i&\in& &\sym^+(W,k) \quad &\text{for }& i\in \mathbb I_{21}
\\
&w_i&\in& \sym^+(W,k), \>&w_i&\in& &\sym^-(W,k) \quad &\text{for }& i\in \mathbb I_{22}
.
\end{alignat*}

All elements with indices from $\mathbb I_{11}$ and $\mathbb I_{22}$ vanish, and
we are done.
\end{proof}

In the particular case $(V,j) = (M_n(K), {}^\top)$ and 
$(W,k) = (\mathbb O, \topbar)$, denoting $J = {}^\top \otimes \topbar$, and 
taking into account that $\mathbb O^+ = K1$, we get:
\begin{equation}\label{eq-dec}
\sym^+(M_n(\mathbb O),J) \simeq 
M_n^+(K) \otimes 1 \dotplus M_n^-(K) \otimes \mathbb O^- .
\end{equation}
(In the case where $n=3$, so
$\sym^+(M_3(\mathbb O,J))$ is the $27$-dimensional exceptional simple Jordan 
algebra, this decomposition was noted in \cite[\S 3.3]{draper}).

In particular,
$$
\dim \sym^+(M_n(\mathbb O,J)) = \frac{n(n+1)}{2} + 7 \cdot \frac{n(n-1)}{2} =
4n^2 - 3n .
$$

For any $m,s\in M_n^-(K)$, we have
$$
(m \otimes 1) \circ (s \otimes 1) = (m \circ s) \otimes 1 ,  
$$
what implies that $M_n^+(K) \otimes 1$ is a (Jordan) subalgebra of 
$\sym^+(M_n(\mathbb O,J))$. Moreover, for any $x,y \in M_n^-(K)$, and 
$a \in \mathbb O^-$, we have:
\begin{alignat*}{6}
&(m \otimes 1) \>&\circ&\> (x \otimes a) \>&=& &(m \circ x) &\otimes a
\\
&(x \otimes a) \>&\circ&\> (y \otimes a) \>&=& -N(a) &(x \circ y) &\otimes  1 .
\end{alignat*}

It follows that $M_n^+(K) \otimes 1 \dotplus M_n^-(K) \otimes a$ is a subalgebra
of $\sym^+(M_n(\mathbb O,J))$; let us denote this subalgebra by 
$\mathscr L^+(a)$. We have an isomorphism of Jordan algebras 
$\mathscr L^+(a) \otimes_K \overline K \simeq M_n(\overline K)$, where 
$\overline K$ is the quadratic extension of $K$; the isomorphism is provided by
sending $m \otimes 1$ to $m$ for $m \in M_n^+(\overline K)$, and $x \otimes a$ 
to $\sqrt{-N(a)}\,x$ for $x \in M_n^-(\overline K)$.

Further,
$$
(M_n^+(K) \otimes 1) \circ (M_n^-(K) \otimes \mathbb O^-) \subseteq 
M_n^-(K) \otimes \mathbb O^-  .
$$

On the other hand, the subspace $M_n^-(K) \otimes \mathbb O^-$ is not a 
subalgebra. The formula for multiplication in this subspace in terms of the 
decomposition (\ref{eq-dec}) is obtained using (\ref{eq-o}): for any 
$x,y \in M_n^-(K)$ and $a, b\in \mathbb O^-$, we have
\begin{multline}\label{eq-m}
(x \otimes a) \circ (y \otimes b) = xy \otimes ab + yx \otimes ba = 
\frac 12 (xy+yx) \otimes (ab+ba) + \frac 12 (xy - yx) \otimes (ab - ba) \\ =
N(a,b) \> (x \circ y) \otimes 1 + \frac 12 [x,y] \otimes [a,b] .
\end{multline}

Similarly, we have
\begin{equation}\label{eq-decomp-minus}
\sym^-(M_n(\mathbb O),J) \simeq M_n^-(K) \otimes 1 \dotplus 
M_n^+(K) \otimes \mathbb O^- ,
\end{equation}
and
$$
\dim \sym^-(M_n(\mathbb O),J) = \frac{n(n-1)}{2} + 7 \cdot \frac{n(n+1)}{2} =
4n^2 + 3n.
$$

For any $x,y\in M_n^-(K)$, $m,s\in M_n^+(K)$, and $a \in \mathbb O^-$, we have:
\begin{alignat*}{6}
&[x \otimes 1, y \otimes 1] \>&=&\>      &[&x,y] \>&\otimes&\> 1   \\
&[x \otimes 1, m \otimes a] \>&=&\>      &[&x,m] \>&\otimes&\> a  \\
&[m \otimes a, s \otimes a] \>&=&\> N(a) &[&s,m] \>&\otimes&\> 1 .
\end{alignat*}

It follows that both $M_n^-(K) \otimes 1$ and 
$\mathscr L^-(a) = M_n^-(K) \otimes 1 \dotplus M_n^+(K) \otimes a$ are Lie 
subalgebras of $\sym^-(M_n(\mathbb O),J)$, isomorphic to $\mathfrak{so}_n(K)$, 
and to a form of $\mathfrak{gl}_n(\overline K)$ respectively; the isomorphisms 
are provided by sending $x \otimes 1$ to $x$ for $x \in M_n^-(\overline K)$, and
$m \otimes a$ to $\sqrt{-N(a)}\,m$ for $m\in M_n^+(\overline K)$.

Moreover,
$$
[M_n^-(K) \otimes 1, M_n^+(K) \otimes \mathbb O^-] \subseteq 
M_n^+(K) \otimes \mathbb O^- .
$$

The subspace $M_n^+(K) \otimes \mathbb O^-$ is not a subalgebra: for any 
$m,s \in M_n^+(K)$, $a,b \in \mathbb O^-$, we have
\begin{equation}\label{eq-mult}
[m \otimes a, s \otimes b] = 
\frac 12 (ms - sm) \otimes (ab+ba) + \frac 12 (ms + sm) \otimes (ab - ba) =
\frac{N(a,b)}{2} [m,s] \otimes 1 + (m \circ s) \otimes [a,b] .
\end{equation}

\begin{theorem}\label{th-1}
The algebras $\sym^+(M_n(\mathbb O), J)$ and $\sym^-(M_n(\mathbb O), J)$ are 
simple for any $n \ge 1$.
\end{theorem}

Before we plunge into the proof, a few remarks are in order:
\begin{enumerate}[\upshape(i)]
\item
The cases of $\sym^+(M_n(\mathbb O), J)$ for $n=1,2,3$, and of 
$\sym^-(M_n(\mathbb O), J)$ for $n=1$ are well-known, due to the known structure
of the algebras in question in these cases (see \S \ref{sec-1}); however, our 
proofs, uniform for all $n$, appear to be new. The case of 
$\sym^+(M_4(\mathbb O), J)$ is stated without proof in \cite[Satz 8.1]{ruhaak}.

\item
In \cite{stewart} it is proved that ideals of the tensor product $A \otimes B$
of two algebras $A$ and $B$, where $A$ is central simple, and $B$ satisfies some
other conditions (like having a unit), are of the form $A \otimes I$, where $I$
is an ideal of $B$. In particular, the tensor product of two central simple 
algebras, for example, $M_n(K) \otimes \mathbb O$, is simple. Our method of 
proof of Theorem \ref{th-1}, based on application of the (variant of) Jacobson 
density theorem, resembles those in \cite{stewart}.

\item
A variant of the Jacobson density theorem we will need in our proof concerns 
so-called associative pairs, and is established in \cite[Theorem 1]{pairs}. In 
what follows, we will refer to it as the ``Jacobson density theorem for associative
pairs''. Due to the relations between the spaces of symmetric and skew-symmetric
matrices -- formulas (\ref{eq-inc}) in the particular case $(A,j) = (M_n(K),{}^\top)$ -- the
pair $(M_n^+(K), M_n^-(K))$ whose elements act on each other either via the 
commutator, or via the Jordan multiplication $\circ$, forms an associative 
primitive pair in the terminology of \cite{pairs}.

\item
Another related result about simplicity of nonassociative algebras is 
established in \cite[Satz 5.1]{ruhaak}: the matrix algebra over a composition 
algebra with respect to the Jordan product $\circ$, is simple; a particular case
is the algebra $(M_n(\mathbb O), \circ)$.

\end{enumerate}

\begin{proof}[Proof of Theorem \ref{th-1} in the case of 
$\sym^+(M_n(\mathbb O),J)$]
Let $I$ be an ideal of $\sym^+(M_n(\mathbb O), J)$. We argue in terms of the 
decomposition (\ref{eq-dec}). Assume first that 
$I \subseteq M_n^-(K) \otimes \mathbb O^-$. Consider an element 
\begin{equation}\label{eq-elem}
\sum_{i=1}^7 x_i \otimes e_i \in I ,
\end{equation}
where $x_i \in M_n^-(K)$, and $e_1,\dots,e_7$ are elements of the standard basis
of $\mathbb O$, as described in \S \ref{sec-1}. For any $y \in M_n^-(K)$, and
any $k=1,\dots,7$, we have 
$$
(y \otimes e_k) \circ (\sum_{i=1}^7 x_i \otimes e_i) = - (x_k \circ y) \otimes 1
+ \text{ terms lying in } M_n^-(K) \otimes \mathbb O^- .
$$
Hence $x_k \circ y = 0$ for any $y \in M_n^-(K)$, and by Lemma \ref{lemma-1},
$y = 0$. This shows that $I=0$, and we may assume 
$I \nsubseteq M_n^-(K) \otimes \mathbb O^-$.

Now take an element $m \otimes 1 + \sum_{i=1}^7 x_i \otimes e_i \in I$, where 
$m \in M_n^+(K)$, $x \ne 0$, and, as previously, $x_i \in M_n^-(K)$. By the 
Jacobson density theorem for associative pairs, for any $m^\prime \in M_n^+(K)$
there is a linear map $R: M_n(K) \to M_n(K)$, formed by a sum of products of the
form $R_{s_1} \dots R_{s_\ell}$, where each $s_i$ belongs to $M_n^+(K)$, and 
$R_s$ is the Jordan multiplication on the element $s$, such that 
$R(m) = m^\prime$ and $R(x_i) = 0$ for any $i=1,\dots,7$. We form the 
corresponding map $\widetilde R$ from the multiplication algebra of 
$\sym^+(M_n(\mathbb O),J)$ by replacing each $R_{s_i}$ by $R_{s_i \otimes 1}$. 
Then $\widetilde R(m \otimes 1) = m^\prime \otimes 1$ and 
$\widetilde R(x_i \otimes e_i) = 0$. Consequently, $m^\prime \otimes 1 \in I$,
and $I$ contains $M_n^+(K) \otimes 1$.

We can write $I$ as the direct sum of vector spaces 
$I = M_n^+(K) \otimes 1 \dotplus S$ for some subspace 
$S \subseteq M_n^-(K) \otimes \mathbb O^-$. As we can obviously form nonzero 
Jordan products between elements of $M_n^+(K) \otimes 1$ and of 
$M_n^-(K) \otimes \mathbb O^-$, we have that $S \ne 0$. Consider again a
nonzero element of $I$ of the form (\ref{eq-elem}). Applying again the Jacobson
density theorem for associative pairs, for any $x \in M_n^-(K)$, and for any 
$k=1,\dots,7$, we get a linear map $R: M_n(K) \to M_n(K)$ generated by Jordan 
multiplications by elements of $M_n^+(K)$, such that $R(x_k) = x$ and 
$R(x_i) = 0$ for $i\ne k$. Deriving from this the map $\widetilde R$ in the 
multiplication algebra of $\sym^+(M_n(\mathbb O),J)$ as above, we get that
$M_n^-(K) \otimes e_k \subseteq I$ for each $k=1,\dots,7$, and hence $I$ 
coincides with the whole algebra $\sym^+(M_n(\mathbb O),J)$.
\end{proof}

\begin{proof}[Proof of Theorem \ref{th-1} in the case of 
$\sym^-(M_n(\mathbb O),J)$]
The proof goes along the same route as in the previous case. 

Let $I$ be an ideal of $\sym^-(M_n(\mathbb O),J)$. Assume first 
$I \subseteq M_n^+(K) \otimes \mathbb O^-$. Consider an element 
\begin{equation}\label{eq-elem1}
\sum_{i=1}^7 m_i \otimes e_i \in I ,
\end{equation}
where $m_i \in M_n^+(K)$. For any $s \in M_n^+(K)$, and any $k=1,\dots,7$, we 
have
$$
[s \otimes e_k, \sum_{i=1}^7 m_i \otimes e_i] = [m_k,s] \otimes 1 + 
\text{terms lying in } M_n^+(K) \otimes \mathbb O^- .
$$
Hence $[m_k,s] = 0$ for any $s\in M_n^+(K)$, and by Lemma \ref{lemma-circ}, 
$m_k = \lambda_k E$ for some $\lambda_k \in K$. Therefore, any element of $I$
is of the form $\sum_{i=1}^7 \lambda_i E \otimes e_k \in E \otimes \mathbb O^-$,
and $I = E \otimes S$ for some subspace $S \subseteq \mathbb O^-$. But then 
$$
[M_n^+(K) \otimes \mathbb O^-, E \otimes S] = M_n^+(K) \otimes [\mathbb O^-,S]
\subseteq E \otimes S ,
$$
what can happen only if all the involved spaces are zero, i.e. $S=0$ and $I=0$.
Therefore, we may assume $I \nsubseteq M_n^+(K) \otimes \mathbb O^-$.

Consider an element $x \otimes 1 + \sum_{i=1}^7 m_i \otimes e_i \in I$, where 
$x \in M_n^-(K)$, $x\ne 0$, and $m_i \in M_n^+(K)$. By the Jacobson density 
theorem for associative pairs, for any $x^\prime \in M_n^-(K)$ there is a linear
map $R: M_n(K) \to M_n(K)$, formed by a sum of products of the form 
$\ad y_1 \dots \ad y_\ell$, where each $y_i$ belongs to $M_n^-(K)$, and $\ad y$
denotes the commutator with $y$, such that $R(x) = x^\prime$, and $R(m_i) = 0$ 
for each $i=1,\dots,7$. Replacing in $R$ each $\ad y_i$ by 
$\ad (y_i \otimes 1)$, we get the map $\widetilde R$ in the multiplication 
algebra of $\sym^-(M_n(\mathbb O),J)$ such that 
$\widetilde R(x \otimes 1) = x^\prime \otimes 1$ and 
$\widetilde R(m_i \otimes e_i) = 0$, and thus 
$$
\widetilde R(x \otimes 1 + \sum_{i=1}^7 m_i \otimes e_i) = x^\prime \otimes 1 .
$$

Consequently, $I$ contains $M_n^-(K) \otimes 1$, and we can write 
$I = M_n^-(K) \otimes 1 \dotplus S$ for some subspace 
$S \subseteq M_n^+(K) \otimes \mathbb O^-$. As $M_n^-(K) \otimes 1$ alone is, 
obviously, not an ideal in $\sym^+(M_n(\mathbb O),J)$, we have $S \ne 0$. 
Consider again a nonzero element of $I$ of the form (\ref{eq-elem1}). By the 
Jacobson density theorem for associative pairs, for any $m \in M_n^+(K)$, and any 
$k=1,\dots,7$, there is a linear map $R: M_n(K) \to M_n(K)$ generated by 
commutators with elements of $M_n^-(K)$, such that $R(m_k) = m$, and 
$R(m_i) = 0$ for $i\ne k$. Deriving from this the map $\widetilde R$ in the
multiplication algebra of $\sym^-(M_n(\mathbb O),J)$ as above, we get that
$\widetilde R(\sum_{i=1}^7 m_i \otimes e_i) = m \otimes e_k$. This shows that
$I$ coincides with the whole algebra $\sym^-(M_n(\mathbb O),J)$.
\end{proof}

\section{$\delta$-derivations}\label{sec-der}

In \cite{petyt}, derivations of the algebras $\sym^+(M_n(\mathbb O),J)$ and
$\sym^-(M_n(\mathbb O),J)$ were computed. Here we extend this result by 
computing $\delta$-derivations of these algebras. Recall that a 
\emph{$\delta$-derivation} of an algebra $A$ is a linear map $D: A \to A$ such
that
\begin{equation}\label{eq-delta}
D(xy) = \delta D(x)y + \delta xD(y)
\end{equation}
for any $x,y \in A$ and some fixed $\delta \in K$. This notion generalizes 
simultaneously the notions of derivation and of centroid (any element of the
centroid is, obviously, a $\frac 12$-derivation).

The set of $\delta$-derivations of an algebra $A$, denoted by $\Der_\delta(A)$,
forms a vector space. Moreover, as noted, for example, in \cite[\S 1]{filippov},
$$
[\Der_\delta(A), \Der_{\delta^\prime}] \subseteq \Der_{\delta\delta^\prime}(A) ,
$$
so the vector space $\Delta(A)$ linearly spanned by all $\delta$-derivations, 
for all possible values of $\delta$, forms a Lie algebra, an extension of the 
Lie algebra $\Der(A)$ of (the usual) derivations of $A$.

\begin{theorem}\label{th-delta}
Let $D$ be a nonzero $\delta$-derivations of the algebra 
$\sym^+(M_n(\mathbb O),J)$ or $\sym^-(M_n(\mathbb O),J)$. Then either 
$\delta = 1$ (i.e., $D$ is a derivation), or $\delta = \frac 12$ and $D$ is a 
multiple of the identity map.
\end{theorem}

The case of $\sym^+(M_n(\mathbb O),J)$ is easier, as the algebra contains a 
unit, and $\delta$-derivations of algebras with unit are tackled by the simple

\begin{lemma}\label{lemma-a}
Let $D$ be a $\delta$-derivation of a commutative algebra $A$ with unit. Then 
either $\delta=1$ (i.e., $D$ is a derivation), or $\delta = \frac 12$ and 
$D = R_a$ for some $a \in A$ such that 
\begin{equation}\label{eq-a}
2(xy)a - (xa)y - (ya)x = 0
\end{equation}
for any pair of elements $x,y \in A$.
\end{lemma}

\begin{proof}
This is, essentially, \cite[Theorem 2.1]{kayg-first} with a bit more (trivial)
details. Repeatedly substituting the unit $1$ in the equality (\ref{eq-delta}) 
gives that either $\delta = 1$ and $D(1) = 0$, or $\delta = \frac 12$ and 
$D(x) = xD(1)$ for any $x\in A$. In the latter case, denoting $D(1) = a$, the 
condition (\ref{eq-delta}) is equivalent to (\ref{eq-a}).
\end{proof}

\begin{proof}[Proof of Theorem \ref{th-delta} in the case of 
$\sym^+(M_n(\mathbb O),J)$]
Due to Lemma \ref{lemma-a} it amounts to description of algebra elements 
satisfying the condition (\ref{eq-a}). Let 
$a = m \otimes 1 + \sum_{i=1}^7 x_i \otimes e_i$ be such an element, where 
$m \in M_n^+(K)$, $x_i \in M_n^-(K)$. Writing the condition (\ref{eq-a}) for the
pair of elements $s \otimes 1$, $t \otimes 1$ where $s,t \in M_n^+(K)$, and 
collecting terms lying in $M_n^+(K) \otimes 1$, we get
$$
2(s \circ t) \circ m - (s \circ m) \circ t - (t \circ m) \circ s = 0
$$
for any $s,t \in M_n^+(K)$. The latter condition means that $R_m$ is a 
$\frac 12$-derivation of the Jordan algebra $M_n^+(K)$, and by 
\cite[Theorem 2.5]{kayg-first}, $m = \lambda E$ for some $\lambda \in K$. As the
set of elements satisfying the condition (\ref{eq-a}) forms a vector space (as,
generally, the set of $\frac 12$-derivations does), by subtracting from $a$ the
element $\lambda E \otimes 1$, we get an element still satisfying the condition
(\ref{eq-a}), so we may assume $\lambda = 0$.

Now writing the condition (\ref{eq-a}) for $a = \sum_{i=1}^7 x_i \otimes e_i$, 
and the pair $x \otimes e_k$, $y \otimes e_\ell$, where $x,y \in M_n^-(K)$ and 
$k,\ell = 1,\dots,7$, $k \ne \ell$, and again collecting terms lying in 
$M_n^+(K) \otimes 1$, we get $[x,y] \circ x_{k*\ell} = 0$. Since 
$[M_n^-(K),M_n^-(K)] = M_n^-(K)$, and the values of $k*\ell$ run over all 
$1,\dots,7$, we get that $M_n^-(K) \circ x_i = 0$ for any $i=1,\dots,7$. By 
Lemma \ref{lemma-1}, $x_i = 0$, what shows that any element 
$a \in \sym^+(M_n(\mathbb O,J))$ satisfying (\ref{eq-a}), is a multiple of the 
unit.
\end{proof}

Before turning to the proof of the $\sym^-(M_n(\mathbb O),J)$ case, we need a 
couple of auxiliary lemmas.

\begin{lemma}\label{lemma-gl}
Let $n>2$. 
\begin{enumerate}[\upshape(i)]
\item
If $\delta \ne 1,\frac 12$, then the vector space 
$\Der_{\delta}(\mathfrak{gl}_n(K))$ is $1$-dimensional, and each 
$\delta$-derivation is a multiple of the map $\xi$ vanishing on 
$\mathfrak{sl}_n(K)$, and sending $E$ to itself.

\item
The vector space $\Der_{\frac 12}(\mathfrak{gl}_n(K))$ is $2$-dimensional, with
a basis consisting of the two maps: the map $\xi$ from part {\rm (i)}, and the 
map coinciding with the identity map on $\mathfrak{sl}_n(K)$, and vanishing on 
$E$.
\end{enumerate}
\end{lemma}

\begin{proof}
This follows immediately from the fact that $\mathfrak{gl}_n(K)$ is the split 
central extension of $\mathfrak{sl}_n(K)$: 
$\mathfrak{gl}_n(K) = \mathfrak{sl}_n(K) \oplus KE$, and the fact, established
in numerous places, that each nonzero $\delta$-derivation of 
$\mathfrak{sl}_n(K)$, $n>2$, is either a usual derivation ($\delta = 1$), or 
element of the centroid ($\delta = \frac 12$) (see, for example,
\cite[Corollary 4.16]{leger-luks} or \cite{filippov}).
\end{proof}

\begin{lemma}\label{lemma-xdm-comm}
Let $D: M_n^+(K) \to M_n^+(K)$ be a linear map such that
\begin{equation}\label{eq-dxm}
D([x,m]) = \delta [x,D(m)]
\end{equation}
for any $x \in M_n^-(K)$, $m \in M_n^+(K)$, and some fixed $\delta \in K$, 
$\delta \ne 0,1$. Then the image of $D$ lies in the one-dimensional linear space
spanned by $E$.
\end{lemma}

\begin{proof}
Replacing in the equality (\ref{eq-dxm}) $x$ by $[x,y]$, where 
$x,y\in M_n^-(K)$, and using the Jacobi identity, we get:
$$
D([x,[y,m]]) - D([y,[x,m]]) = \delta [[x,y],D(m)] .
$$
Using the fact that $[x,m], [y,m] \in M_n^+(K)$, applying again (\ref{eq-dxm}) 
to each term at the left-hand side twice, and using the Jacobi identity, we get
$[[x,y],D(m)] = 0$. Since $[M_n^-(K),M_n^-(K)] = M_n^-(K)$, the latter equality
is equivalent to $[M_n^-(K),D(m)] = 0$. By Lemma \ref{lemma-circ}, $D(m)$ is a 
multiple of $E$ for any $m \in M_n^+(K)$.
\end{proof}

When considering restrictions of $\delta$-derivations to subalgebras, we arrive
naturally at the necessity to consider a more general notion of 
$\delta$-derivations with values in not necessary the algebra itself, but in an 
algebra module. Generally, this require to consider bimodules, but as we will 
need this generalization only in the case of anticommutative (in fact, Lie) 
algebras, we confine ourselves here with the following definition. Let $A$ be 
an anticommutative algebra, and $M$ a left $A$-module, with the module action 
denoted by $\bullet$. A $\delta$-derivation of $A$ with values in $M$ is a 
linear map $D: A \to M$ such that
\begin{equation*}%\label{eq-eqdelta}
D(xy) = - \delta y \bullet D(x) + \delta x \bullet D(y)
\end{equation*}
for any $x,y \in A$.

\begin{proof}[Proof of Theorem \ref{th-delta} in the case of 
$\sym^-(M_n(\mathbb O),J)$]
If $n=1$, the algebra in question is the $7$-di\-men\-si\-o\-nal simple Malcev 
algebra $\mathbb O^-$, and the result is covered by 
\cite[Lemma 3]{filippov-ass}.

Let $n > 2$ and $\delta \ne 1$. As the space of $\delta$-derivations does not 
change under field extensions, we may extend the base field $K$ as we wish, in 
particular, assume that $K$ is quadratically closed.

We may write
\begin{alignat*}{3}
&D(x \otimes 1)   \>&=&\> d(x) \otimes 1   &+& \sum_{i=1}^7 d_i(x) \otimes e_i 
\\
&D(m \otimes e_k) \>&=&\> f_k(m) \otimes 1 &+& \sum_{i=1}^7 f_{ki}(m) \otimes e_i 
\end{alignat*}
for any $x \in M_n^-(K)$, $m\in M_n^+(K)$, $k=1,\dots,7$, and some linear maps
$d: M_n^-(K) \to M_n^-(K)$, $d_i: M_n^-(K) \to M_n^+(K)$, 
$f_k: M_n^+(K) \to M_n^-(K)$, and $f_{ki}: M_n^+(K) \to M_n^+(K)$.

For a fixed $k=1,\dots,7$, consider the Lie subalgebra 
$$
\mathscr L^-(e_k) = M_n^-(K) \otimes 1 \dotplus M_n^+(K) \otimes e_k
$$
of $\sym^-(M_n(\mathbb O),J)$, isomorphic, as noted in \S \ref{sec-simp}, to 
$\mathfrak{gl}_n(K)$. According to decomposition (\ref{eq-decomp-minus}), 
$\sym^-(M_n(\mathbb O),J)$ is decomposed, as an $\mathscr L^-(e_k)$-module, into
the direct sum of the adjoint module $\mathscr L^-(e_k)$, and the module 
$M_n^+(K) \otimes B_k$ (note, however, that the latter is not a Lie module). 
This implies that the restriction of $D$ to $\mathscr L^-(e_k)$, being composed 
with the canonical projection $\sym^-(M_n(\mathbb O),J) \to \mathscr L^-(e_k)$, 
i.e., the map
\begin{alignat*}{3}
&x \otimes 1   \>&\mapsto&\> d(x)   \otimes 1 \>&+&\> d_k(x) \otimes e_k   \\
&m \otimes e_k \>&\mapsto&\> f_k(m) \otimes 1 \>&+&\> f_{kk}(m) \otimes e_k , 
\end{alignat*}
is a $\delta$-derivation of $\mathscr L^-(e_k)$ (with values in the adjoint 
module). 

Denote by $SM_n(K)$ the space of matrices from $M_n^+(K)$ with trace zero (so 
$M_n^+(K) = SM_n(K) \oplus KE$). By Lemma \ref{lemma-gl}, either 
$\delta \ne \frac 12$, and each such map is of the form
\begin{alignat*}{3}
&x \otimes 1   \>&\mapsto&\> 0  \\
&m \otimes e_k \>&\mapsto&\> 0, \quad m\in SM_n(K) \\
&E \otimes e_k \>&\mapsto&\> \mu_k E \otimes e_k
\end{alignat*}
for some $\mu_k \in K$; or $\delta = \frac 12$, and each such map is of the form
\begin{alignat*}{3}
&x \otimes 1   \>&\mapsto&\> \lambda_k x \>&\otimes&\> 1   \\
&m \otimes e_k \>&\mapsto&\> \lambda_k m \>&\otimes&\> e_k, \quad m\in SM_n(K) 
\\
&E \otimes e_k \>&\mapsto&\> \mu_k E     \>&\otimes&\> e_k
\end{alignat*}
for some $\lambda_k, \mu_k \in K$. Taking into account that one of these 
alternatives holds uniformly for all values of $k$, we arrive at two cases:

\emph{Case 1}. $\delta \ne 1, \frac 12$ and $D(M_n^-(K) \otimes 1) = 0$.

\emph{Case 2}. $\delta = \frac 12$, and $D(x \otimes 1) = \lambda x \otimes 1$ 
for any $x \in M_n^-(K)$ and some fixed $\lambda \in K$.

Moreover, in both cases 
$D(M_n^+(K) \otimes \mathbb O^-) \subseteq M_n^+(K) \otimes \mathbb O^-$. We 
will handle these two cases together, keeping in mind that $\lambda = 0$ if
$\delta \ne \frac 12$.

Consider now the restriction of $D$ to $M_n^+(K) \otimes \mathbb O^-$. Since 
\begin{equation*}
\Hom(M_n^+(K) \otimes \mathbb O^-, M_n^+(K) \otimes \mathbb O^-)
\simeq
\Hom(M_n^+(K),M_n^+(K)) \otimes \Hom(\mathbb O^-,\mathbb O^-) ,
\end{equation*}
we may write
$$
D(m \otimes a) = \sum_{i \in \mathbb I} d_i(m) \otimes \alpha_i(a)
$$
for any $m \in M_n^+(K)$, $a \in \mathbb O^-$, and some linear maps
$d_i: M_n^+(K) \to M_n^+(K)$, $\alpha_i: \mathbb O^- \to \mathbb O^-$. Writing 
the condition of $\delta$-derivation (\ref{eq-delta}) for pair $x \otimes 1$, 
$m \otimes a$, where $x \in M_n^-(K)$, $m \in M_n^+(K)$, 
$a \in \mathbb O^-$, we get
\begin{equation}\label{eq-22}
\sum_{i\in \mathbb I}
\Big(d_i([x,m]) - \delta [x,d_i(m)]\Big) \otimes \alpha_i(a)
= \delta\lambda [x,m] \otimes a .
\end{equation}

In Case 1 the right-hand side of (\ref{eq-22}) vanishes and hence we may assume
$d_i([x,m]) = \delta [x,d_i(m)]$ for any $x\in M_n^-(K)$, $m\in M_n^+(K)$, and 
any $i\in \mathbb I$. By Lemma \ref{lemma-xdm-comm}, each $d_i(m)$ is a multiple
of $E$, and hence 
$D(M_n^+(K) \otimes \mathbb O^-) \subseteq E \otimes \mathbb O^-$. But then
writing (\ref{eq-delta}) for pair $m\otimes a$, $s\otimes b$, where 
$m,s\in M_n^+(K)$, $a,b\in \mathbb O^-$, and taking into account 
(\ref{eq-mult}), yields $D((m \circ s) \otimes [a,b]) = 0$. Since 
$(M_n(K),\circ)$ and $(\mathbb O^-, \liebrack)$ are perfect (in fact, simple)
algebras, the latter equality implies vanishing of $D$ on the whole 
$M_n^+(K) \otimes \mathbb O^-$, and thus on the whole 
$\sym^-(M_n(\mathbb O),J)$, a contradiction.

Hence we are in Case 2, and $\delta = \frac 12$. Setting in this case 
$d_\star = -\lambda \id_{M_n^+(K)}$, and $\alpha_\star = \id_{\mathbb O^-}$, the
equality (\ref{eq-22}) can be rewritten as
$$
\sum_{i\in \mathbb I \cup \{\star\}}
\Big(d_i([x,m]) - \frac 12 [x,d_i(m)]\Big) \otimes \alpha_i(a) = 0 .
$$

As in the previous case, this means that there are new linear maps 
$\widetilde d_i$, $\widetilde \alpha_i$ which are linear combinations of $d_i$ 
and $\alpha_i$, respectively, and such that 
\begin{equation}\label{eq-tilde}
\sum_{i \in \mathbb I \cup \{\star\}} 
\widetilde d_i \otimes \widetilde \alpha_i = 
\sum_{i \in \mathbb I \cup \{\star\}} d_i \otimes \alpha_i ,
\end{equation}
and $\widetilde d_i([x,m]) = \frac 12 [x, \widetilde d_i(m)]$. 
Lemma~\ref{lemma-xdm-comm} tells us, as previously, that each 
$\widetilde d_i(m)$ is a multiple of $E$, and hence the image of the map at the
left-hand side of (\ref{eq-tilde}) lies in $E \otimes \mathbb O^-$. Since the 
right-hand side of (\ref{eq-tilde}) is equal to 
$D + d_\star \otimes \alpha_\star$, we have
$$
D(m \otimes a) = \lambda m \otimes a + E \otimes \beta(m,a)
$$
for any $m \in M_n^+(K)$, $a \in \mathbb O^-$, and some bilinear map 
$\beta: M_n^+(K) \times \mathbb O^- \to \mathbb O^-$. Replacing $D$ by the
$\frac 12$-derivation $D - \lambda \id$, we arrive at the situation as in the
previous case: a $\delta$-derivation (with $\delta = \frac 12$) vanishing on 
$M_n^-(K) \otimes 1$, and taking values in $E \otimes \mathbb O^-$ on 
$M_n^+(K) \otimes \mathbb O^-$. Hence $D - \lambda \id$ vanishes on the whole 
$\sym^-(M_n(\mathbb O),J)$, and $D = \lambda \id$, as claimed.

Finally, consider the case $n=2$. In this case Lemma \ref{lemma-gl} is not 
true: in addition to the cases described there, there is the $5$-dimensional 
space of $(-1)$-derivations of $\mathfrak{sl}_2(K)$, and thus the corresponding 
$6$-dimensional space of $(-1)$-derivations of $\mathfrak{gl}_2(K)$ (see 
\cite[Example 1.5]{hopkins} or \cite[Example in \S 3]{filippov-5}). In view of
this, to proceed like in the proof of the case $n>2$, considering 
$\delta$-derivations of the Lie subalgebras $\mathscr L^-(e_k)$, would be too 
cumbersome, and we are taking a slightly alternative route.

Denote by $H = \left(\begin{matrix} 0 & 1 \\ -1 & 0 \end{matrix}\right)$ the 
basic element of the $1$-dimensional space $M_2^-(K)$. Consider the subalgebra 
$E \otimes \mathbb O^-$ of $\sym^+(M_2(\mathbb O),J)$, isomorphic to the 
$7$-dimensional simple Malcev algebra $\mathbb O^-$. As an 
$E \otimes \mathbb O^-$-module, $\sym^+(M_2(\mathbb O),J)$ decomposes as the 
direct sum of the trivial $1$-dimensional module $KH \otimes 1$, and the module
$M_2^+(K) \otimes \mathbb O^-$, which is isomorphic to the direct sum of $3$ 
copies of the adjoint module ($\mathbb O^-$ acting on itself). Thus $D$, being 
restricted to $E \otimes \mathbb O^-$, is equal to the sum of a 
$\delta$-derivation with values in the trivial module, which is obviously zero,
and $3$ $\delta$-derivations of $\mathbb O^-$. By the result mentioned at the 
beginning of this proof, the latters are zero in the case 
$\delta \ne 1,\frac 12$, and are multiples of the identity map in the case 
$\delta = \frac 12$. Consequently, $D(E \otimes a) = m_0 \otimes a$ for any 
$a \in \mathbb O^-$, and some fixed $m_0 \in M_2^+(K)$.

Now write
$$
D(H \otimes 1) = \lambda H \otimes 1 + \sum_{i=1}^7 m_i \otimes e_i  
$$
for some $\lambda \in K$, and $m_i \in M_2^+(K)$. Writing the condition of
$\delta$-derivation (\ref{eq-delta}) for pair $H \otimes 1$, $E \otimes e_k$,
$k=1,\dots,7$, we get
$$
\pm 2\sum_{1\le i \le 7, i\ne k} m_i \otimes e_{i*k} + [H,m_0] \otimes e_k = 0 .
$$
It follows that $m_i = 0$ for any $i=1,\dots,7$, and
$D(H \otimes 1) = \lambda H\otimes 1$.

Now let
$$
D(m \otimes a) = \beta(m,a) H \otimes 1 + 
\text{terms lying in } M_2^+(K) \otimes \mathbb O^-
$$
for any $m\in M_2^+(K)$, $a\in \mathbb O^-$, and some bilinear map 
$\beta: M_2^+(K) \otimes \mathbb O^- \to K$. Writing the condition of 
$\delta$-derivation for pair $H \otimes 1$, $m \otimes a$, and collecting terms
which are multiples of $H \otimes 1$, we get $\beta(m,a) H \otimes 1 = 0$. Thus
$D(M_2^+(K) \otimes \mathbb O^-) \subseteq M_2^+(K) \otimes \mathbb O^-$, and we
may proceed as in the generic case $n>2$ above.
\end{proof}

Note that it is also possible to pursue the case $\delta = 1$ along the same 
lines, what would give an alternative proof of the results of \cite{petyt}, as
well as of the classical result that derivation algebra of the $27$-dimensional
exceptional simple Jordan algebra is isomorphic to the simple Lie algebra of 
type $F_4$.

There is a vast literature devoted to $\delta$-derivations of algebras and 
related notions (see \cite{hopkins}, \cite{filippov-5}--\cite{filippov-ass}, 
\cite{kayg-first}, \cite{leger-luks} for a small but representative sample). Our
strategy to prove Theorem \ref{th-delta} was to identify certain Lie subalgebras
of the algebra $\sym^-(M_n(\mathbb O),J)$, and consider $\delta$-derivations of
those subalgebras with values in the whole $\sym^-(M_n(\mathbb O),J)$. 
Developing further the methods of the above cited papers, it is possible to 
prove that $\delta$-derivations of semisimple Lie algebras of classical type 
with coefficients in finite-dimensional modules are either (inner) derivations,
or multiples of the identity map on irreducible constituents of the module 
isomorphic to the adjoint module of the algebra, or, in the case of the direct 
summands in the algebra isomorphic to $\mathfrak{sl}_2(K)$, $(-1)$-derivations 
with values in the irreducible constituents isomorphic to the adjoint 
$\mathfrak{sl}_2(K)$-modules. This general fact would allow to simplify further the proof of 
Theorem \ref{th-delta}, but establishing it will require considerable (though pretty much straightforward) 
efforts, and will lead us far away from the topic of this paper. We hope to return to this elsewhere.

As by \cite{petyt}, both $\Der(\sym^+(M_n(\mathbb O),J))$ for $n\ge 4$ and
$\Der(\sym^-(M_n(\mathbb O),J))$ for any $n$ are isomorphic to the Lie algebra 
$G_2 \oplus \mathfrak{so}_n(K)$, then by Theorem \ref{th-delta}, both
$\Delta(\sym^+(M_n(\mathbb O),J))$ and $\Delta(\sym^-(M_n(\mathbb O),J))$ are 
isomorphic to the one-dimensional trivial central extension of 
$G_2 \oplus \mathfrak{so}_n(K)$.

Finally, note an important

\begin{corollary}
The algebras $\sym^+(M_n(\mathbb O),J)$ and $\sym^-(M_n(\mathbb O),J)$ are 
central simple.
\end{corollary}

\begin{proof}
By Theorem \ref{th-1}, these algebras are simple, and by Theorem \ref{th-delta}
their centroid coincides with the ground field.
\end{proof}

\section{Symmetric associative forms}\label{sec-form}

Let $A$ be an algebra. A bilinear symmetric form $\varphi: A \times A \to K$ is
called \emph{associative}, if
\begin{equation}\label{eq-f}
\varphi(xy,z) = \varphi(x,yz)
\end{equation}
for any $x,y,z \in A$. (In the context of Lie algebras, associative forms are
usually called \emph{invariant}, because in that case the condition (\ref{eq-f})
is equivalent to invariance of the form $\varphi$ with respect to the standard 
action of the underlying Lie algebra on the space of symmetric bilinear forms).

For a matrix $X = (a_{ij})$ from $M_n(\mathbb O)$, by $\overline X$ we will 
understand the matrix $(\overline{a_{ij}})$, obtained by element-wise 
application of conjugation in $\mathbb O$.

\begin{theorem}\label{th-form}
Any bilinear symmetric associative form on $\sym^+(M_n(\mathbb O),J)$ or on 
$\sym^-(M_n(\mathbb O),J)$ is a scalar multiple of the form 
\begin{equation}\label{eq-form}
(X, Y) \mapsto \Tr(XY + \overline X \, \overline Y) .
\end{equation}
\end{theorem}

The form (\ref{eq-form}) is reminiscent of the Killing form on simple Lie 
algebras of classical type (and \emph{is} the Killing form when restricted from
the algebra $\sym^-(M_n(\mathbb O,J))$ to its Lie subalgebra 
$\mathfrak{so}_n(K)$, see below).

\begin{proof}
According to Corollary in \S \ref{sec-der}, both algebras are central simple. 
The standard linear-algebraic arguments show that any bilinear symmetric 
associative form on a simple algebra is nondegenerate, and that any two 
nondegenerate symmetric associative forms on a central algebra are proportional 
to each other. Thus, the vector space of bilinear symmetric associative forms on
a central simple algebra is either $0$- or $1$-dimensional.

Now it remains to observe that in both cases this space is $1$-dimensional by 
verifying that the form (\ref{eq-form}) is indeed associative. The most 
convenient way to do this is, perhaps, to rewrite the form in terms of decompositions (\ref{eq-dec}) or (\ref{eq-decomp-minus}). On the algebra 
$\sym^+(M_n(\mathbb O),J)$ we obtain
\begin{alignat*}{3}
&(m \otimes 1, s \otimes 1) &\>\mapsto\>& 2\Tr(ms)                \\
&(m \otimes 1, x \otimes a) &\>\mapsto\>& 0                       \\
&(x \otimes a, y \otimes b) &\>\mapsto\>& (ab+ba) \Tr(xy) ,
\end{alignat*}
and on $\sym^-(M_n(\mathbb O),J)$, 
\begin{alignat*}{3}
&(x \otimes 1, y \otimes 1) &\>\mapsto\>& 2\Tr(xy)                \\
&(x \otimes 1, m \otimes a) &\>\mapsto\>& 0                       \\
&(m \otimes a, s \otimes b) &\>\mapsto\>& (ab+ba) \Tr(ms) .
\end{alignat*}
Here, as usual, $x,y\in M_n^-(K)$, $m,s\in M_n^+(K)$, and $a,b \in \mathbb O^-$.
(For the algebra $\sym^+(M_n(\mathbb O),J)$, the associativity follows also from
\cite[Satz 5.2]{ruhaak}, where it is proved that the form (\ref{eq-form}) is a 
symmetric associative form on a larger algebra $(M_n(\mathbb O), \circ)$).
\end{proof}

Note that it is possible to get an alternative, direct proof of 
Theorem~\ref{th-form} without appealing to results of \S \ref{sec-der}, in the 
linear-algebraic spirit of the proofs of Proposition in \S \ref{sec-simp}, or of
Theorem \ref{th-delta}.

\section{Further questions}\label{sec-q}

1) 
To compute automorphism group of algebras $\sym^+(M_n(\mathbb O),J)$ and
$\sym^-(M_n(\mathbb O),J)$. Are they isomorphic to $G_2 \times SO(n)$?

\smallskip

2)
For $n>3$, the algebras $\sym^+(M_n(\mathbb O),J)$ are no longer Jordan. How 
``far'' they are from Jordan algebras? Which identities these algebras do 
satisfy? A starting point could be investigation of (non-Jordan)
representations of the Jordan subalgebras which are forms of the full matrix 
Jordan algebra $M_n(K)$, mentioned in \S \ref{sec-simp}, in the whole 
$\sym^+(M_n(\mathbb O),J)$.

\smallskip

3)
What one can say about subalgebras of the algebras in question? Say, what are 
the maximal subalgebras? Maximal Jordan subalgebras of 
$\sym^+(M_n(\mathbb O),J)$? 

\smallskip

4) 
We have considered the case of split octonions only, but, as considerations of 
the Jordan algebra $\sym^+(M_3(\mathbb O),J)$ suggest, the case of arbitrary 
octonion algebras poses further challenges. In particular, one may wish to 
characterize the algebras $\sym^+(M_n(\mathbb O,J))$ for split $\mathbb O$ in 
some internal way (like cubic Jordan algebras in the case $n=3$), and then 
characterize their forms as $\sym^+(M_n(\mathbb O,J))$ for arbitrary 
$\mathbb O$.

\smallskip

5)
Investigate the case of characteristic $3$. Though this case is, perhaps, of 
little interest for physics, in characteristic $3$ the $7$-dimensional algebra
$\mathbb O^-$ is not merely a Malcev algebra, but isomorphic to the Lie algebra
$\mathfrak{psl}_3(K)$ (see, for example, \cite[Theorem 4.26]{elduque-kochetov}).
This suggests that the algebras $\sym^+(M_3(\mathbb O),J)$ and 
$\sym^-(M_3(\mathbb O),J)$ in this characteristic may satisfy a different set of
identities than in the generic case, perhaps, more tractable and more closer to the classical 
identities (Lie, Jordan, etc.).

\section*{Acknowledgements}

Thanks are due to Francesco Toppan who pointed us to \cite{toppan} and 
\cite{ruhaak}, and explained the importance of algebras considered here.
GAP \cite{gap} was utilized to check some of the computations performed in this
paper.

\end{document}